\documentclass{article}
\usepackage{amsmath,amssymb,amsfonts}
\usepackage{graphicx}
\usepackage{color}

\newtheorem{theorem}{Theorem}

\newtheorem{lemma}{Lemma}

\newcommand{\Div}{\operatorname{div}}
\def\dO{\partial\Omega }
\begin{document}

\title{A connection between filter stabilization and eddy viscosity models}
\author{Maxim A. Olshanskii \thanks{
Department of Mathematics, University of Houston, Houston, TX 77204-3008 and Dept. Mechanics and Mathematics, Moscow State
University, Moscow 119899
email: \texttt{molshan@math.uh.edu}; partially supported by the RFBR grants 12-01-00283, 11-01-00767, 12-01-91330.}  \and Xin Xiong \thanks{ Department of Mathematics, University of Pittsburgh, Pittsburgh 15260, PA,  USA; email: \texttt{xix21@pitt.edu}; partially supported by the NSF under grant DMS-0810385} }
\date{}
\maketitle

\begin{abstract}
Recently, a new approach for the stabilization of the incompressible
Navier-Stokes equations for higher Reynolds numbers was introduced based on
the nonlinear differential filtering of solutions on every time step of a discrete scheme.
In this paper, the stabilization is shown to be equivalent to a certain eddy-viscosity
model in LES. This allows a refined analysis and further understanding of
desired filter properties. We also consider the application of the
filtering in a projection (pressure correction) method, the standard
splitting algorithm for time integration of the incompressible fluid
equations. The paper proves an estimate on the  convergence of the filtered numerical
solution to the  corresponding DNS solution.
\end{abstract}

\section{Introduction}

A stabilization of a numerical time-integration algorithm for the
incompressible Navier-Stokes equations
\begin{equation}
\left.
\begin{array}{rcl}
u_t +(u \cdot \nabla) u - \nu \Delta u + \nabla p & = & f \\
\operatorname{ div} u & = & 0%
\end{array}
\right. \quad \text{in}~~\Omega \times (0,T]~,  \label{NS}
\end{equation}
for large Reynolds numbers with the help of an additional filtering step was
recently introduced in \cite{LRT10}. Denote by $w^{n}$ or $u^n$
approximations to the Navier-Stokes system velocity solution at time $t_n$,
and similarly $p^n$ approximates pressure $p(t_n)$. Let $\triangle t =t_{n+1}-t_n
$. The algorithm, referred to further as (A1),  reads: For $n=0,1,\dots$ and $u^0=u(t^0)$

\begin{description}
\item {1.} compute intermediate velocity $w^{n+1}$ from
\begin{equation*}
\left\{
\begin{split}
\frac1{\triangle t}( w^{n+1}- u^n) + ( w^{n+1}\cdot \nabla) w^{n+1} + \nabla
p^{n+1} - \nu\Delta w^{n+1} & = f^{n+1}, \\
\operatorname{ div} w^{n+1} & = 0,
\end{split}
\right.
\end{equation*}
subject to appropriate boundary conditions;

\item {2.} filter the intermediate velocity, $\overline{ w^{n+1}}:= F\,
w^{n+1}$;

\item {3.} relax $u^{n+1}:= (1-\chi) w^{n+1}+\chi\overline{ w^{n+1}}$, with
a relaxation parameter $\chi\in[0,1]$.
\end{description}

Here $F$ is a generic nonlinear filter acting from $L^{2}(\Omega )^{3}$ to $%
H^{1}(\Omega )^{3}$. We shall consider further in the paper several examples
of differential filters. The convergence of the finite element solutions of
(A1) to the smooth Navier-Stokes solution has been analyzed in \cite{LRT10}. One advantage of
the approach is the convenience of  implementation  within an existing CFD code
for laminar flows and flexibility in the choice of a filter.
Numerical results from \cite{BRTT,ELN09,LRT10,LRT11a,LayTak} with composite nonlinear differential filters, as defined in Section~\ref{sec_diff}, consistently show more
precise localization of model viscosity and its more precise correlation with the action
of nonlinearity on the smallest resolved scales than plain Smagorinsky type LES or VMS methods.
Thus we deem the approach deserves further study, should be put into perspective and related to developing LES models.

In this paper, we show that introducing the filter stabilization is closely
related (and even equivalent in a sense which is made precise further in the
paper) to adapting a certain eddy-viscosity model for LES. The
connection to a LES model helps us to quantify the model dissipation
introduced by the filter stabilization (Theorem~\ref{Th1}), formulate stability criteria (see \eqref{stab_cond2} and \eqref{stab_cond}), and gives
 insight into the choice of the filter and the relaxation parameter.
In particular, it provides an explanation why the stabilization by
the filtering avoids adding excessive model viscosity in regions
of larger velocity gradients, unlike most other eddy viscosity models.

The entire approach is specifically designed for treating higher
Reynolds number flows. Therefore, it is natural to extend it to the Chorin-Temam-Yanenko
type splitting algorithms, which are the prevailing method for
the time-integration of the incompressible Navier-Stokes equations for
fast unsteady flows. Such (rather natural) extension is presented in the paper
together with the relevant error analysis.
We note right away that the analysis demonstrates the convergence of
numerical solutions to the Navier-Stokes
smooth solution, while it would be also interesting to  analyze the error of the
numerical solutions to a (presumably smoother) solution of the corresponding LES model.
However, the specific difficulty we faced in the latter case is the lacking of the monotone
property by most of  eddy viscosity indicator functionals, which were numerically proved to
be useful in defining the filter $F$, see Section~\ref{sec_diff}.
Though practically attractive, introducing such functionals makes the mathematical well-posedness
of the LES model and accordingly the error analysis hard to accomplish and  we are unaware
of relevant results in this direction.


\section{Filter stabilization and LES model}

\label{sec2}

It is well known, see, e.g., \cite{Boyd} or \cite{MLF03}, that explicit filtering is related to
adding eddy or artificial viscosity.
The connection of the filter stabilization as defined above to LES modeling is easily recovered by
noting that shifting the index $n+1\rightarrow n$ on steps 2 and 3 and using
step 1 gives the implicit discretization of the Navier-Stokes equations, with explicitly treated
nonlinear dissipation term:
\begin{equation}
\left\{
\begin{split}
\frac{1}{\triangle t}(w^{n+1}-w^{n})+(w^{n+1}\cdot \nabla )w^{n+1}+\nabla
p^{n+1}-\nu \Delta w^{n+1}+\frac{\chi }{\triangle t}G  w^{n}& =f^{n+1},
\\
\Div w^{n+1}& ={0},
\end{split}%
\right.   \label{NS1}
\end{equation}%
with
\begin{equation*}
G:=I-F,\qquad \text{$I$ is the identity operator}.
\end{equation*}%
Assume $\chi =\chi _{0}\triangle t$, where $\chi _{0}$ is a time- and
mesh-independent constant, then \eqref{NS1} can be treated as the
time-stepping scheme for
\begin{equation}
\left\{
\begin{split}
w_{t}+(w\cdot \nabla )w+\nabla p-\nu \Delta w+\chi _{0}G\,w& =f , \\
\Div w& ={0}.
\end{split}%
\right.   \label{NS1cont}
\end{equation}%
These arguments show that the  numerical integrator (A1) with filter stabilization is
the splitting scheme for solving \eqref{NS1cont}. Furthermore,
\eqref{NS1cont} can be observed as a LES model, with $\chi _{0}G\,w$
corresponding to the Reynolds stress tensor closure:
\begin{equation*}
\nabla \cdot (\overline{w\otimes w}-\overline{w}\otimes \overline{w})\approx
\chi _{0}G\,w.
\end{equation*}%
This simple observation leads to a refined analysis and better interpretation of
the numerical results and the method properties.

We note that  $\chi =O(\triangle t)$ is exactly the scaling of relaxation parameter which
allows us to prove optimal convergence result for a time-stepping splitting method (Theorem~\ref{ThMain}). Furthermore,
numerical experiments in \cite{ELN09,Fischer}
suggested that $\chi =O(\triangle t)$ is indeed the right scaling of the
relaxation parameter with respect to numerical solution accuracy.

We start by showing several numerical properties of the approach.
Throughout the paper we use $(\cdot,\cdot)$ and $\|\cdot\|$ to denote the $%
L^2$ scalar product and the norm, respectively. For the sake of analysis,
assume the homogeneous Dirichlet boundary conditions for velocity. Taking
the $L^2$ scalar product of \eqref{NS1} with $2\triangle t w^{n+1}$ and
integrating by parts gives
\begin{equation}  \label{rel1}
\| w^{n+1}\|^2-\| w^{n}\|^2+\frac12\| w^{n+1}- w^{n}\|^2 + \nu\triangle
t\|\nabla w^{n+1}\|^2+ \chi( G w^{n}, w^{n+1}) = \triangle t ( f^{n+1}, w^{n+1}).
\end{equation}
For a self-adjoint filtering operator, i.e. $(Gu,v)=(Gv,u)$ for any $u,v\in
H^1_0(\Omega)^3$, the equality \eqref{rel1} can be alternatively written as
\begin{multline}  \label{rel1c}
\| w^{n+1}\|^2-\| w^{n}\|^2 + \nu\triangle t\|\nabla w^{n+1}\|^2+
\frac\chi2\left(( G w^{n+1}, w^{n+1})+( G w^{n}, w^{n})\right) \\
= \triangle t ( f, w^{n+1})+\frac12\left(\chi( G (w^{n+1}- w^{n}),w^{n+1}- w^{n}) -\|
w^{n+1}- w^{n}\|^2\right).
\end{multline}
Considering the last two terms on the right-hand side, we immediately get
the sufficient condition of the energy stability of \eqref{NS1} for the case
of self-adjoint filters:
\begin{equation}  \label{stab_cond2}
\chi(G u, u)\le\|u\|^2\quad \forall ~ u\in H^1_0(\Omega)^3.
\end{equation}
If $G$ is not necessarily  self-adjoint, one may rewrite \eqref{rel1} as
\begin{multline*}
\| w^{n+1}\|^2-\| w^{n}\|^2+\frac12\| w^{n+1}- w^{n}\|^2 + \nu\triangle
t\|\nabla w^{n+1}\|^2+ \chi( G w^{n}, w^{n})
= \triangle t ( f, w^{n+1})+\chi( G w^{n}, w^{n}- w^{n+1}).
\end{multline*}
Thanks to the Cauchy inequality one gets for any $\theta\in \mathbb{R}$:
\begin{multline}  \label{rel2}
\| w^{n+1}\|^2-\| w^{n}\|^2 + \nu\triangle t\|\nabla w^{n+1}\|^2+
(1-\theta)\chi (G w^{n}, w^{n}) \\
\le \triangle t( f, w^{n+1})-\chi\Big(\theta(G w^{n}, w^{n})- \frac\chi2( G
w^{n}, G w^{n})\Big).
\end{multline}
In this more general case, one may consider the following sufficient
condition for the energy stability. Fixing, for example, $\theta=\frac12$,
assures the sum of the last two terms in \eqref{rel2} is positive if
\begin{equation}  \label{stab_cond}
\chi(G u, G u)\le( Gu, u)\quad \forall ~u\in H^1_0(\Omega)^3.
\end{equation}

Assume $G$ is self-adjoint and $w^n$ approximates a smooth in time
Navier-Stokes solution, then \eqref{rel1c} leads to the following energy
balance relation of the numerical method:
\begin{equation*}
\| w^{N}\|^2+ \nu\sum_{n=1}^N\triangle t\|\nabla w^{n}\|^2+
\chi_0\sum_{n=1}^N \triangle t(G w^{n}, w^{n}) = \| w^{0}\|^2+\sum_{n=1}^N
\triangle t(f^{n}, w^{n}) + O(\triangle t).
\end{equation*}
In particular, we may conclude that the filter stabilization introduces the \textit{%
model dissipation} of
\begin{equation}  \label{model_diff}
\chi_0\sum_{n=1}^N \triangle t(G w^{n}, w^{n}).
\end{equation}

Finally, we notice that the filtering and relaxation steps in (A1) can be rearranged as
\[
\frac{u^{n+1}-w^{n+1}}{\triangle t} =-\chi_0G\,w^{n+1},
\]
which is the explicit Euler method for integrating
\begin{equation}\label{expltFilter}
u_t=-\chi_0G\,u~~\mbox{on}~[t_n,t_{n+1}],~~\text{with}~u(t_n)=w(t_{n+1}).
\end{equation}
The coupling of a DNS method with the evolution equation  \eqref{expltFilter} is known as another way of introducing explicit filtering
in modelling of dynamical systems, e.g. \cite{Boyd}. This suggests that an improvement leading
to higher order methods for integrating \eqref{expltFilter} might be possible.

In the next section, we shall study properties of the operator $G$ for a class of nonlinear differential filters.


\section{Nonlinear Differential Filters}\label{sec_diff}

Linear differential filters have a long history in LES, see \cite{Fischer2}.   We also point to \cite{Holm} and  references therein for applications of linear differential filters in the Lagrange-averaging turbulence models.
In this section, we consider a family of \textit{nonlinear} differential filters for the filtering procedure. Some conclusions will be drawn
concerning the stability conditions \eqref{stab_cond2}, \eqref{stab_cond}
and equivalence to other approaches in the LES modelling. We use the
following notation:
\begin{equation*}
V:=\left\{ v\in H_{0}^{1}(\Omega)^{3}\,:\,\Div v=0\right\} ,\quad
H=\left\{ v\in \;L^{2}(\Omega )^{3}:\Div v=0, v\cdot n|_{\dO}=0\right\} .
\end{equation*}%
By $\mathbb{P}$ we denote the $L^{2}$ orthogonal projector from $%
L^{2}(\Omega )^{3}$ onto $H$.

For a given sufficiently smooth vector function $u$ and $w\in L^2(\Omega)^3$ we define $F\, w$ as the
solution to
\begin{equation}  \label{diff_filter}
(\delta^2 a(u)\nabla(F\, w),\nabla v)+(F\, w, v)=( w, v)\quad \forall v\in X,
\end{equation}
with an indicator functional $0\le a(u)\le 1$ and filtering radius $\delta^2$,
which generally may depend on $x$ and $t$, $\delta_{\max}=\max_{x,t}|\delta|$. Here $X= H^1_0(\Omega)^3$ or
$X= V$, if the filter is div-free preserving.
We note that it is not immediately clear if the problem \eqref{diff_filter}
is well-posed. In practice, this is not an issue, since in a finite
dimension setting, e.g. for a finite element method, the bilinear form from
the left-hand side of \eqref{diff_filter} is elliptic and thus \eqref{diff_filter} is
well-posed. Otherwise, we may assume $0<\varepsilon\le a(u)\le 1$ for some
sufficiently small positive $\varepsilon$. If we assume this, none of our results
further in the paper depend on the parameter $\varepsilon$.  It is standard
to base the indicator functional on the input function $w$ itself, that is $%
u=w$ and we will denote $\overline{w}:= F\, w$ in this case. However, in the
course of analysis we need to consider (auxiliary) filtering with $u\neq w$.
If we need to show explicitly the function used for the indicator, we shall
write $F(u)  w$ instead of $F\, w$ or  $F(w)  w$ instead of $%
\overline{w}$.

The action of $G=I-F$, $w_g:=G\, w$, is defined formally as the solution to
\begin{equation}  \label{G_filter}
( \delta^2 a(u) \nabla w_g,\nabla v)+( w_g, v)=( \delta^2 a(u)\nabla
w,\nabla v)\quad \forall v\in X.
\end{equation}

The operator $G$ is self-adjoint on $X$ and in the operator notation it can
be written as
\begin{equation}  \label{G_operator}
G= -\left[I-\Delta_a\right]^{-1}\Delta_a,
\end{equation}
with
\begin{equation*}
\Delta_a:= \left\{%
\begin{split}
\operatorname{ div}(\delta^2 a(u)\nabla)&\quad \text{if}~ X= H^1_0(\Omega)^3, \\
\mathbb{P}\operatorname{ div}(\delta^2 a(u)\nabla) &\quad \text{if}~ X= V.
\end{split}
\right.
\end{equation*}

Since operator $\Delta_a$ is self-adjoint   and positive definite, one see from %
\eqref{G_operator} that $G\le I$ and thus the sufficient stability condition %
\eqref{stab_cond2} holds for any $\chi\in[0,1]$. This can be easily verified
in a formal way by substituting $v= F\, w$ in \eqref{diff_filter} to get $%
(w,F\, w)\ge0$ and thus $(w, G w)=(w,w-F\, w)\le\|w\|^2$ for any $w\in H_{0}^{1}(\Omega)^{3}$. Moreover, varying $%
\theta$ in \eqref{rel2} and using \eqref{stab_cond}, one  shows
the energy stability estimate for any $\chi\in[0,2]$. However, such
refinement is not important for our further analysis. \medskip

With the help of \eqref{model_diff} and \eqref{G_operator}, we now quantify the model dissipation introduced by the differential filters. To make notation
shorter and without loss of generality, let $\chi=\chi_0 \triangle t$.

First, representation \eqref{G_operator} immediately implies $G\le -\Delta_a$.
Thus the additional dissipation introduced by the differential filtering
does not exceed those introduced by the LES closure model:
\begin{equation}  \label{closure}
\Div(\overline{ w\otimes w}-\overline{ w}\otimes\overline{ w})\approx
-\chi_0\Delta_a w.
\end{equation}
It is easy to show that for a discrete case and if the condition
\begin{equation*}
\delta\lesssim~ \text{spatial mesh width}
\end{equation*}
holds and $0\le a(u)\le 1$, then the dissipation introduced by the differential filtering %
\eqref{diff_filter} is \textit{equivalent} to the dissipation of the closure
model \eqref{closure}.

We make the above statement more precise for a finite element
discretization. To this end, assume a consistent triangulation $\mathcal{T}$
of $\Omega$, satisfying the minimal angle condition
\begin{equation*}
\inf_{K\in\mathcal{T}}\rho(K)/r(K)=:\alpha_0>0
\end{equation*}
where $\rho(K)$ and $r(K)$ are the diameters of inscribed and superscribed
circles (spheres in 3D) for a triangle (tetrahedron) $K$. We have the
following result.

\begin{theorem}
\label{Th1}  Assume $X$ is the finite element space of continuous functions
which are polynomials of degree $p\ge1$ on every element $K$ and $\max_{x\in
K}|\delta(x)|\le C_\delta\,r(K)$ for any $K\in\mathcal{T}$, with a constant $%
C_\delta$ independent of $K$. Then for any $w\in X$ the equivalence
\begin{equation}  \label{FEequive}
\widetilde{c}\,(\delta^2 a(u)\nabla w,\nabla w)\le (G\,w,w) \le (\delta^2
a(u)\nabla w,\nabla w)
\end{equation}
holds with a constant $\widetilde{c}>0$ independent of $w$, the indicator $%
a(\cdot)$, and the filtering radius $\delta$. The constant $\widetilde{c}>0$
may depend on $p$, $C_\delta$, and $\alpha_0$.
\end{theorem}

\begin{proof}
Consider the finite element inverse inequality
\begin{equation}\label{inverse}
\|\nabla w\|_{L^2(K)}\le c_0\rho(K)^{-1} \| w\|_{L^2(K)}\qquad\forall\,w\in X,
\end{equation}
where the constant $c_0$ depends only on the polynomial degree $p$ and $\alpha_0$. The inequality \eqref{inverse}, the assumption on $\delta$ and the minimal angle condition imply
\begin{equation}\label{aux21}
\|\delta\nabla w\|_{L^2(K)}\le \widetilde{C} \| w\|_{L^2(K)},
\end{equation}
where the constant $\widetilde{C}$ depends only on $p$, $C_\delta$, and $\alpha_0$. Squaring \eqref{aux21},  summing over all $K\in\mathcal{T}$, and recalling that $a(\cdot)\le1$, implies
\begin{equation}\label{aux11}
(\delta^2a(u)\nabla w,\nabla w)\le \widetilde{C}^2 \| w\|^2.
\end{equation}
Denote $w_g=G\,w$ for some $w\in X$. We set $v=w_g$ and $v=-w$ in  \eqref{G_filter} and sum up the equalities
to get
\begin{align*}
0&=( \delta^2 a(u) \nabla w_g,\nabla w_g)+( w_g, w_g)-2( \delta^2 a(u)\nabla w,\nabla w_g)
-( w_g, w)+( \delta^2 a(u)\nabla w,\nabla w)\\ &=\|w_g\|^2-( w_g, w)+( \delta^2 a(u)\nabla (w-w_g),\nabla(w-w_g)).
\end{align*}
Thus, it holds $\|w_g\|^2\le( w_g, w)$, i.e. the condition \eqref{stab_cond}. Now we set $v=w$ in  \eqref{G_filter}
and use \eqref{stab_cond} and \eqref{aux11} to estimate
\begin{align*}
( \delta^2 a(u) \nabla w,\nabla w)& =( \delta^2 a(u) \nabla w_g,\nabla w)+( w_g, w)\\
&\le\frac12( \delta^2 a(u) \nabla w_g,\nabla w_g)+\frac12( \delta^2 a(u) \nabla w,\nabla w)+( w_g, w)\\
&\le\frac12\widetilde{C}^2 \| w_g\|^2+\frac12( \delta^2 a(u) \nabla w,\nabla w)+( w_g, w)\\
&\le(\frac12\widetilde{C}^2+1)( w_g, w)+\frac12( \delta^2 a(u) \nabla w,\nabla w).
\end{align*}
We proved the lower bound in \eqref{FEequive}.

To show the upper bound we set $v=w_g$ and $v=w$ in  \eqref{G_filter} and sum up the equalities
to get
\[
0=( \delta^2 a(u) \nabla w_g,\nabla w_g)+( w_g, w_g)
+( w_g, w)-( \delta^2 a(u)\nabla w,\nabla w).
\]
This yields the upper bound  in \eqref{FEequive}: $( w_g, w)\le( \delta^2 a(u)\nabla w,\nabla w)$.

\end{proof}

Few conclusions can be drawn from the equivalence result \eqref{FEequive}
concerning the relation of the filter stabilization to some other eddy-viscosity models.

The use of the linear differential filter ($a\equiv 1$), as considered in \cite{ELN09},
is equivalent to the method of artificial viscosity. This means that the
model dissipation is equivalent to the isotropic diffusion scaled with $\chi
_{0}\delta ^{2}$. Given what is known about the method of artificial
viscosity, it is not surprising  that the method is not very accurate in this
case. Thus, more elaborated indicator functionals should be used.
Generally, we may think of $a(u)$ as a real valued  functional, depending on $%
u,\nabla u$, and  selected with the intent that
\begin{align*}
a(u(x))& \approx 0\quad \text{for\  laminar\  regions\  or\  persistent\
flow\  structures}, \\
a(u(x))& \approx 1\quad \text{for\  flow\  structures\  which\  decay\
rapidly}.
\end{align*}

The choice of the Smagorinsky type indicator function, $a(u)=|\nabla u|$,
does not necessarily satisfy the condition $a(u)\leq 1$. In this case, we
do not have the equivalence result of the filter stabilization to the Smagorinsky
LES model. Only the upper bound in \eqref{FEequive} is guaranteed to hold.
Thus the dissipation introduced by the filtering with $a(u)=|\nabla u|$
is likely \textit{less} than that of the Smagorinsky model. This can be a
desirable property, since the Smagorinsky LES model is known to be severely
over-diffusive for certain flows, e.g. \cite{Sagaut}, and several ad hoc corrections were introduced
such as van Driest damping, dynamic models, and others, see  \cite{Driest,Germano,Piomelli}.

Several reasonable indicator functions $a(u)$ are known to satisfy the
boundedness condition: $0\le a(u)\le1$. These are the re-normalized Smagorinsky
type indicator~\cite{BIR09}, the indicator based on the $Q$-criteria~\cite%
{WHM} and the Vreman indicators  \cite{Vreman}; also an indicator based on
the normalized helical density distribution was considered in \cite{BRTT}. Given
several indicators $a_i(\cdot)$, $i=1,\dots,N$, the combined indicator can
be defined as the geometric mean: $a(\cdot):=\left(\prod\limits_{i=1}^Na_i(%
\cdot)\right)^{\frac1N}$.

We remark, that the convergence results proved further in this paper do \textit{not}
rely on any smoothness properties or particular form of $a(\cdot)$.

The last remark in this section is that Theorem~\ref{Th1} does not give much
insight if enforcing the divergence constraint in the filter is important or
not. However, if we assume $X=V$ in \eqref{diff_filter}, i.e., the filtered
velocity satisfies the divergence free condition, then this slightly simplifies
the error analysis in Section~\ref{sec_est}.

\section{Projection scheme with filter stabilization}

One idea behind introducing the filter stabilization or explicit filtering
was to provide CFD software users and developers with a simple way to enhance existing codes for laminar
incompressible flows to compute high Reynolds number flows. This goal is
accomplished by making the filtering procedure algorithmically independent
of a time integration method. Driven by this intention, we consider the
Chorin \cite{Chorin:68} splitting (projection) scheme with the additional
separate filtering step. Projection methods are the common numerical approach to
the incompressible Navier-Stokes equations and form a family of splitting
algorithms, cf. \cite{GMS,Prohl}. We perform the numerical analysis for the
simplest first order method given below. From the algorithmic standpoint,
the generalization to higher order projection methods is straightforward,
although analysis may become considerably more involved.

Projection methods split the time evolution of the velocity vector field
according to the momentum equation and the projection of the velocity to
satisfy the divergence-free condition. The filtering step can be introduced
before or after the projection step. In the former case, it is not necessary to augment
the filter with the div-free constraint, since the projection step takes
care of the keeping the approximates in the subspace of div-free functions.
If the filter is div-free preserving, then it is reasonable to put it after the projection.
In this paper we consider the constrained filter.
We shall study the following algorithm:

\begin{description}
\item {Step 1:} Solve the convection-diffusion type problem: Given $u^n$, $%
w^{\ast}$, find $\widetilde {w^{n+1}}$:
\begin{equation}  \label{step1}
\left\{
\begin{split}
\frac1{\triangle t}( \widetilde {w^{n+1}}- u^n) + ( w^{\ast}\cdot \nabla)
\widetilde {w^{n+1}} - \nu\Delta \widetilde {w^{n+1}} & = f^{n+1}, \\
\widetilde {w^{n+1}}|_{\partial\Omega }&= 0.
\end{split}
\right.
\end{equation}
The velocity $w^{\ast}$ is typically an interpolation from previous times,
e.g. $w^{\ast}:= w^{n}$ or higher order interpolation. For the sake of
analysis we consider $w^{\ast}= w^{n}$.

\item {Step 2:} Project $\widetilde {w^{n+1}}$ on the div-free subspace: Find $%
p^{n+1}$ and $w^{n+1}$ solving the Neumann pressure Poisson problem:
\begin{equation}  \label{step2}
\left\{
\begin{split}
\frac1{\triangle t}( w^{n+1}- \widetilde {w^{n+1}}) + \nabla p^{n+1} & = 0, \\
\operatorname{ div} w^{n+1} & = {0}, \\
n\cdot w^{n+1}|_{\partial\Omega }&= 0.
\end{split}
\right.
\end{equation}

\item {Step 3:} Filter: $\overline{ w^{n+1}}:= F\, w^{n+1}$;\newline

\item {Step 4:} Relax:
\begin{equation}  \label{step4}
u^{n+1}:= (1-\chi) w^{n+1}+\chi\overline{ w^{n+1}},
\end{equation}
with some $\chi\in[0,1]$.
\end{description}

Similar to what was shown in section~\ref{sec2}, shifting the index $n+1\to n
$ on steps 2--4 and substituting into \eqref{step1} gives for $%
\chi=\chi_0\triangle t$

\begin{equation}  \label{NS1pr}
\left\{
\begin{split}
\frac1{\triangle t}( \widetilde {w^{n+1}}- \widetilde {w^{n}}) + ( w^{\ast}\cdot
\nabla) \widetilde {w^{n+1}} + \nabla p^{n+1} - \nu\Delta \widetilde {w^{n+1}} +
\chi_0 G   \widetilde {w^{n}} - \triangle t \chi_0 G  \nabla p^{n+1}&
= f^{n+1}, \\
\operatorname{ div} \widetilde {w^{n+1}} -\triangle t \Delta p^{n+1}& = {0}.
\end{split}
\right.
\end{equation}

From \eqref{NS1pr} we see that the splitting scheme \eqref{step1}--%
\eqref{step4} is formally  the first order accurate time-discretization of the LES
model \eqref{NS1cont}.\smallskip

Further, we show that the splitting scheme \eqref{step1}--\eqref{step4} is
stable. There are two well-known approaches to accomplish the error analysis
of projection methods. The one of Rannacher and Prohl \cite{Prohl}, \cite%
{Rannacher} uses the relation between projection and quasi-compressibility
methods as it is seen from \eqref{NS1pr}. However, this analysis needs
considerable effort to get extended to equations different from the plain
Navier-Stokes equations. Another framework is mainly due to Shen (see \cite%
{JS,Shen2}), where convergence results were shown based on energy type estimates.
In our error analysis we follow (to a certain extent) arguments
from these two papers.

\section{Stability}

To show the stability of the splitting scheme, we need the following simple
auxiliary result:

\begin{lemma}
\label{lem_aux} For $w^{n+1}$ and $u^{n+1}$ from the algorithm~\eqref{step1}%
--\eqref{step4} and the filter $F$ defined in \eqref{diff_filter}, it holds
\begin{equation*}
\|w^{n+1}\|\ge\|u^{n+1}\|.
\end{equation*}
\end{lemma}

\begin{proof} From the definition \eqref{diff_filter} we obtain:
\begin{equation*}
(\delta^{2}a(w^{n+1})\nabla\overline{w^{n+1}},\nabla\overline{w^{n+1}})+\|\overline{w^{n+1}}\|^{2}=(w^{n+1},\overline{w^{n+1}})
=\frac{1}{2}(\|w^{n+1}\|^{2}+\|\overline{w^{n+1}}\|^{2}-\|w^{n+1}-\overline{w^{n+1}}\|^{2}).
\end{equation*}
This yields
\begin{equation*}
\|w^{n+1}\|^{2}= 2(\delta^{2}a(w^{n+1})\nabla\overline{w^{n+1}}, \nabla\overline{w^{n+1}})+\|\overline{w^{n+1}}\|^{2}+\|\overline{w^{n+1}}-w^{n+1}\|^{2}.\label{sta3}
\end{equation*}
Hence, $\|w^{n+1}\|\ge\|\overline{w^{n+1}}\|$.
From \eqref{step4}, we get
\begin{equation*}
\|u^{n+1}\|\le(1-\chi)\|w^{n+1}\|+\chi\|\overline{w^{n+1}}\|\le\|w^{n+1}\|\quad\text{for}~\chi\in[0,1].
\end{equation*}

\end{proof}

Denote by $\|\cdot\|_{-1}$ the $L^2$-dual norm for $H^1_0(\Omega)^3$. Now we are ready to prove the following  stability  result.

\begin{theorem}
\label{LemStab} The algorithm~\eqref{step1}--\eqref{step4} is stable in the
sense of the following a priori estimate:
\begin{equation}  \label{eqStab}
\|w^{l}\|^{2}+\sum_{n=0}^{l-1}\|w^{n+1}-\widetilde{w^{n+1}}%
\|^{2}+\sum_{n=0}^{l-1}\|\widetilde{w^{n+1}}-u^{n}\|^{2}+\sum_{n=0}^{l-1}\nu%
\triangle t\|\nabla\widetilde{w^{n+1}}\|^{2}\leq\|w^{0}\|^{2}+%
\sum_{n=0}^{l-1}\nu^{-1}\triangle t\|f(t_{n+1})\|_{-1}^{2}
\end{equation}
for any $l=1,2,\dots$.
\end{theorem}

\begin{proof}

Take the $L^2$ scalar product of \eqref{step1} with $2\triangle t\widetilde{w^{n+1}}$:
\[
2(\widetilde{w^{n+1}}-u^{n},\widetilde{w^{n+1}})+2\nu\triangle t\|\nabla\widetilde{w^{n+1}}\|^{2}=2\triangle t( f^{n+1},\widetilde{w^{n+1}}) \leq \nu^{-1}\triangle t\|f^{n+1}\|_{-1}^{2}+\nu\triangle t\|\nabla\widetilde{w^{n+1}}\|^{2}.
\]
Rewriting and simplifying this leads to:
\begin{equation}
\|\widetilde{w^{n+1}}\|^{2}-\|u^{n}\|^{2}+\|\widetilde{w^{n+1}}-u^{n}\|^{2}+\nu\triangle t\|\nabla\widetilde{w^{n+1}}\|^{2}\leq \nu^{-1}\triangle t\|f^{n+1}\|_{-1}^{2}.\label{sta1}
\end{equation}
The $L^2$ scalar of \eqref{step2} with $2\triangle t\, w^{n+1}$ and $\Div\,w^{n+1}=0$ gives
\begin{equation*}
2(w^{n+1}-\widetilde{w^{n+1}},w^{n+1})=0\quad\Longrightarrow\quad\|w^{n+1}\|^{2}-\|\widetilde{w^{n+1}}\|^{2}+\|w^{n+1}-\widetilde{w^{n+1}}\|^{2}=0.
\end{equation*}
Substituting $\|\widetilde{w^{n+1}}\|^{2}$ with $\|w^{n+1}\|^{2}+\|w^{n+1}-\widetilde{w^{n+1}}\|^{2}$ in \eqref{sta1} yields
\begin{equation*}
\|w^{n+1}\|^{2}-\|u^{n}\|^{2}+\|w^{n+1}-\widetilde{w^{n+1}}\|^{2}+\|\widetilde{w^{n+1}}-u^{n}\|^{2}+\nu\triangle t\|\nabla\widetilde{w^{n+1}}\|^{2}\leq \nu^{-1}\triangle t\|f^{n+1}\|_{-1}^{2}.
\end{equation*}
The application of Lemma~\ref{lem_aux} gives
\begin{equation*}
(\|w^{n+1}\|^{2}-\|w^{n}\|^{2})+\|w^{n+1}-\widetilde{w^{n+1}}\|^{2}+\|\widetilde{w^{n+1}}-u^{n}\|^{2}+\nu\triangle t\|\nabla\widetilde{w^{n+1}}\|^{2}\leq \nu^{-1}\triangle t\|f^{n+1}\|_{-1}^{2}.
\end{equation*}
Summing up the inequality from $n=0,\dots, l-1$, we arrive at \eqref{eqStab}.
\end{proof}

\section{Error Estimates}

\label{sec_est}

We shall use $\langle \cdot,\cdot\rangle $ to denote the duality product
between $H^{-s}$ and $H_{0}^{s}(\Omega)$ for all $s\geq 0$. In the
following, we assume that the given data and solution to the equations %
\eqref{NS} subject to the homogeneous Dirichlet velocity boundary conditions
satisfy
\begin{equation}
\begin{cases}
u_0 \in (H^2(\Omega))^d \cap V, \\
{f} \in L^{\infty}(0,T;(L^2(\Omega))^d) \cap L^{2}(0,T;(H^1(\Omega))^d), \\
{f}_t \in L^2(0,T;H^{-1}), \\
\sup_{t\in [0,T]} \|\nabla u (t) \| \le \tilde{C}.%
\end{cases}
\label{A}
\end{equation}
We will use $c$ and $C$ as a generic positive constant which may depend on $%
\Omega,\nu,T$, constants from various Sobolev inequalities, $u_{0}$, ${f}$,
and the solution $u$ through the constant $\tilde{C}$ in \eqref{A}. \\[0.8em]
Under the assumption \eqref{A} one can prove the following inequalities, cf.
\cite{HR}:
\begin{eqnarray}
\sup_{t\in [0,T]} \{ \| u (t) \|_2 + \| u_t(t) \| + \| \nabla p(t) \| \} \le
C,  \label{en1} \\
\int_{0}^{T} \| \nabla u_t (t) \|^2 + t\| u_{tt}\|^2dt \le C,  \label{en2}
\end{eqnarray}
which will be used in the sequel. Further we often use the following
well-known \cite{Temam2} estimates for the bilinear form $b( u, {v}, {w}%
)=\int_{\Omega} ( u \cdot \nabla ) {v} \cdot {w}\, \mathrm{d} {x} $:
\begin{equation*}  
b( u, {v}, {w})\le \left\{%
\begin{array}{l}
c\| \nabla u\|\|\nabla {v}\|^{\frac12}\| {v}\|^{\frac12}\|\nabla {w}\|, \\
c\| u\|_2\| {v}\|\|\nabla {w}\|, \\
c\|\nabla u\|\| {v}\|_2\| {w}\|.%
\end{array}%
\right.
\end{equation*}
and $b( u, {v}, {w})=-b( u, {w}, {v})$ for $u\in H$.

Define the Stokes operator $A u = -\mathbb{P} \Delta u, \,\,\, \forall\, u \in
D(A)=V \cap H^{2}(\Omega)^{3} $. We will use the following
properties: $A$ is an unbounded positive self-adjoint closed operator in $H$
with domain $D(A)$, and its inverse $A^{-1}$ is compact in $H$ and satisfies
the following relations~\cite{JS,Shen2}:
\begin{equation*}
\exists\, c,C > 0,\,\,\, \text{such that} \,\,\, \forall u \in H :
\left\lbrace
\begin{array}{l}
\| A^{-1} u \|_2 \le c \| u \|~~\mbox{and}~~\| A^{-1} u \| \le c \| u
\|_{V^{\prime }}, \\[0.8em]
c\| u\|^{2}_{V^{\prime }} \le (A^{-1} u, u) \le C\| u \|^{2}_{V^{\prime }}.%
\end{array}
\right.  
\end{equation*}

Before we proceed with the error analysis, we prove several auxiliary
results given below in Lemma~\ref{lemma1}. The lemma gives estimates on the
difference between a velocity $w$ and the filtered velocity $F(u) w$.

\begin{lemma}
\label{lemma1} Consider the differential filter $F$ defined in %
\eqref{diff_filter} with some sufficiently smooth vector function $u$. For any $w\in V$ and $F w\in V$
it holds
\begin{align}
\|w- F {w}\|_{\phantom{V'}}&\leq \delta_{\mathrm{max}}\|\nabla w\|,
\label{estFilt1} \\
\|w- F {w}\|_{V^{\prime }}&\leq \delta^{2}_{\mathrm{max}}\|\nabla w\|.
\label{estFilt2}
\end{align}
\end{lemma}

\begin{proof}
Denote $e=w- F {w}$. The  equation \eqref{diff_filter} gives
\begin{equation*}
(\delta^{2}a(u)\nabla e,\nabla
v)+(e,v)=(\delta^{2}a(u)\nabla w, \nabla v)\;
\forall\;v\in\;V.
\end{equation*}
Letting $v=e$ yields
\begin{align*}
\|\delta\sqrt{a(u)}\nabla e\|^{2}+\|e\|^{2}&=(\delta^{2}a(u)\nabla
w), \nabla e)\leq \|\delta\sqrt{a(u)}\nabla w\|\|\delta\sqrt{a(u)}\nabla e\|\\
&\le
\|\delta\sqrt{a(u)}\nabla e\|^2+\frac14\|\delta\sqrt{a(u)}\nabla w\|^2
\le
\|\delta\sqrt{a(u)}\nabla e\|^2+\frac14\delta^{2}_{\rm max}\|\nabla w\|^2.
\end{align*}
This proves \eqref{estFilt1}.
To show \eqref{estFilt2}, we note that setting $v= F\,
{w}-w$ in \eqref{diff_filter} gives
\begin{equation*}
(\delta^{2}a(u)\nabla F\,
{w},\nabla( F {w}-w))=-\| F {w}-w\|^{2}\leq 0.
\end{equation*}
Hence, we obtain:
\begin{equation}
\|\delta\sqrt{a(u)}\nabla F {w}\|^{2}\leq \|\delta\sqrt{a(u)}\nabla w\|^{2}.\label{fne2}
\end{equation}
Allowing $v=A^{-1}(w- F {w})$ in \eqref{diff_filter}
 leads to the following relations:
\begin{align*}
\|w- F {w}\|_{V'}^{2}&=(w- F {w},A^{-1}(w- F {w}))=(\delta^{2}a(u)\nabla F\,
{w},\nabla A^{-1}(w- F {w}))\\
&\leq\|\delta^{2}a(u)\nabla F {w}\|\|\nabla A^{-1}(w- F {w})\|\leq\frac{1}{2}(\|\delta^{2}a(u)\nabla F {w}\|^{2}+\|w- F {w}\|_{V'}^{2})\\
&\leq\frac{1}{2}\delta^{2}_{\rm max}\|\delta\sqrt{a(u)}\nabla F {w}\|^{2}+\frac{1}{2}\|w- F {w}\|_{V'}^{2}.
\end{align*}
The last estimate and  \eqref{fne2} implies \eqref{estFilt2}.
\end{proof}

Further in this section, we show that $\overline{w^{n+1}}, w^{n+1}$
and $u^{n+1}$ are all strongly  $O((\triangle t)^\frac{1}{2}+\delta)$ approximations to $%
u(t_{n+1})$ in $L^{2}(\Omega)^{3}$ provided $\chi=\chi_{0}\triangle t$.
Then we use this result to improve the error estimates to weakly $O(\triangle t+\delta^2)$
approximations. This analysis largely follows the framework from \cite{JS} and \cite{Shen2} for the pure (non-filtered) Navier-Stokes equations, so we shall
refer to these papers and \cite{OST} for some arguments which do not depend
on the filtering procedure.

\begin{lemma}
\label{lemma_est1} Let $u$ be the solution to the Navier-Stokes system,
satisfying \eqref{A}. Denote
\begin{equation*}
\widetilde{\epsilon^{n+1}}=u(t_{n+1})-\widetilde{w^{n+1}},\;\;%
\epsilon^{n+1}=u(t_{n+1})-w^{n+1},\;\;and\;\;e^{n+1}=u(t_{n+1})-u^{n+1}.
\end{equation*}
The following estimate holds
\begin{equation}  \label{firstEst}
\|\widetilde{\epsilon^{l}}\|^{2}+\sum_{n=0}^{l-1}(\|\epsilon^{n+1}-%
\widetilde{\epsilon^{n+1}}\|^{2}+\|\widetilde{\epsilon^{n+1}}%
-e^{n}\|^{2})+\sum_{n=0}^{l-1}2\nu\triangle t\|\nabla\widetilde{%
\epsilon^{n+1}}\|^{2}\leq C(\triangle t +\delta_{\max}^2).
\end{equation}
\end{lemma}

\begin{proof}
Let $R^{n}$ denote the truncation error defined by
\begin{equation}
\frac{1}{\triangle t}(u(t_{n+1})-u(t_{n}))-\nu\triangle u(t_{n+1})+(u(t_{n+1})\cdot\nabla)u(t_{n+1})+\nabla p(t_{n+1})=f^{n+1}+R^{n},\label{true}
\end{equation}
where $R^{n}$ is the integral residual of the Taylor series, i.e,
\begin{equation*}
R^{n}=\frac{1}{\triangle t}\int_{t_{n}}^{t_{n+1}}(t-t_{n})u_{tt}(t)dt.
\end{equation*}
By subtracting \eqref{step1} from \eqref{true}, we obtain
\begin{equation}
\frac{1}{\triangle t}(\widetilde{\epsilon^{n+1}}-e^{n})-\nu\triangle\widetilde{\epsilon^{n+1}}=(w^{n}\cdot\nabla)\widetilde{w^{n+1}}-(u(t_{n+1})\cdot\nabla)u(t_{n+1})
-\nabla p(t_{n+1})+R^{n}.\label{er1}
\end{equation}
Taking the $L^2$ scalar product of \eqref{er1} with $2\triangle t\widetilde{\epsilon^{n+1}}$, we get
\begin{multline}
\|\widetilde{\epsilon^{n+1}}\|^{2}-\|e^{n}\|^{2}+\|\widetilde{\epsilon^{n+1}}-e^{n}\|^{2}+2\nu\triangle t\|\nabla\widetilde{\epsilon^{n+1}}\|^{2}=2\triangle t(R^{n},\widetilde{\epsilon^{n+1}})-2\triangle t(\nabla p(t_{n+1}),\widetilde{\epsilon^{n+1}})\\+2\triangle tb^{*}(w^{n},\widetilde{w^{n+1}},\widetilde{\epsilon^{n+1}})-2\triangle tb^{*}(u(t_{n+1}),u(t_{n+1}),\widetilde{\epsilon^{n+1}}).\label{er2}
\end{multline}
The terms on the right-hand side are bounded exactly  the same way as in \cite{JS} p.64 and \cite{Shen2} p.512, leading to the estimates:
\begin{equation}
\triangle t |b^{*}(w^{n},\widetilde{w^{n+1}},\widetilde{\epsilon^{n+1}})- b^{*}(u(t_{n+1}),u(t_{n+1}),\widetilde{\epsilon^{n+1}})| \leq \frac{\nu\triangle t}{2}\|\nabla\widetilde{\epsilon^{n+1}}\|^{2}+C\triangle t\|\epsilon^{n}\|^{2}+C(\triangle t)^{2}\int_{t_{n}}^{t_{n+1}}\|u_{t}\|^{2}dt, \label{non}
\end{equation}
\begin{equation}
2\triangle t(R^{n},\widetilde{\epsilon^{n+1}})\leq\frac{\nu\triangle t}{4}\|\nabla\widetilde{\epsilon^{n+1}}\|^{2}+C(\triangle t)^{2}\int_{t_{n}}^{t_{n+1}}t\|u_{tt}\|_{-1}^{2}dt\label{r},
\end{equation}
\begin{equation}
2\triangle t(\nabla p(t_{n+1}),\widetilde{\epsilon^{n+1}})=2\triangle t(\nabla p(t_{n+1}),\widetilde{\epsilon^{n+1}}-e^{n})\leq\frac{1}{2}\|\widetilde{\epsilon^{n+1}}-e^{n}\|^{2}+2(\triangle t)^{2}\|\nabla p(t_{n+1})\|^{2}.\label{pre}
\end{equation}
Combining the inequalities \eqref{er2}, \eqref{non}, \eqref{r}, \eqref{pre}, and rearranging terms, we obtain
\begin{multline}
\|\widetilde{\epsilon^{n+1}}\|^{2}-\|e^{n}\|^{2}+\frac{1}{2}\|\widetilde{\epsilon^{n+1}}-e^{n}\|^{2}+\nu\triangle t\|\nabla\widetilde{\epsilon^{n+1}}\|^{2}\\\leq2(\triangle t)^{2}\|\nabla p(t_{n+1})\|^{2}+C\triangle t\|\epsilon^{n}\|^{2}+C(\triangle t)^{2}(\int_{t_{n}}^{t_{n+1}}t\|u_{tt}\|_{-1}^{2}dt+\int_{t_{n}}^{t_{n+1}}\|u_{t}\|^{2}dt).\label{er3}
\end{multline}
The step 4 of the algorithm~\eqref{step1}--\eqref{step4} yields
\begin{equation}
e^{n}=(1-\chi)\epsilon^{n}+\chi F(w^{n+1}) {\epsilon^{n}}+\chi(u(t_{n})- F(w^{n+1}) {u(t_{n})}).\label{et}
\end{equation}
The definition of the filter and recalling that $\epsilon^n$ is the $L^2$ projection of $\widetilde{\epsilon^{n}}$ give $\| F(w^{n+1}) {\epsilon^{n}}\|\leq \|\epsilon^{n}\|\le\|\widetilde{\epsilon^{n}}\|$. We use this to deduce from \eqref{et} the following estimate:
\begin{align*}
\|e^{n}\|=(1-\chi)\|\epsilon^{n}\|+\chi\| F(w^{n+1}) {\epsilon^{n}}\|+\chi\|u(t_{n})- F(w^{n+1}) {u(t_{n})}\|
\le\|\widetilde{\epsilon^{n}}\|+\chi\|u(t_{n})- F(w^{n+1}) {u(t_{n})}\|.
\end{align*}
Now we apply \eqref{estFilt1} and square the resulting inequality to  get (for the sake of convenience
we assume $\triangle t \le C$ and recall $\chi=\chi_0\triangle t $):
\begin{equation}\label{e_by_eps}
\|e^{n}\|^2\le(1+\triangle t )\|\widetilde{\epsilon^{n}}\|^2+C\triangle t  \delta_{\max}^2.
\end{equation}
We substitute \eqref{e_by_eps} to the left-hand side of \eqref{er3} for $\|e^{n}\|$, use $ \|\epsilon^{n}\|\le\|\widetilde{\epsilon^{n}}\|$ and arrive at
\begin{multline}
\|\widetilde{\epsilon^{n+1}}\|^{2}-\|\widetilde{\epsilon^{n}}\|^{2}+\|\epsilon^{n+1}-\widetilde{\epsilon^{n+1}}\|^{2}+\frac{1}{2}\|\widetilde{\epsilon^{n+1}}-e^{n}\|^{2}+\nu\triangle t\|\nabla\widetilde{\epsilon^{n+1}}\|^{2}\\ \le 2(\triangle t)^{2}\|\nabla p(t_{n+1})\|^{2}+C\triangle t\|\widetilde{\epsilon^{n}}\|^{2} +C(\triangle t)^{2}\left(\int_{t_{n}}^{t_{n+1}}t\|u_{tt}\|_{-1}^{2}dt +\int_{t_{n}}^{t_{n+1}}\|u_{t}\|^{2}dt\right)
+C\triangle t  \delta_{\max}^2.\label{er5}
\end{multline}
Summing up \eqref{er5} from $n=0$ to $n=l-1$, assuming that $\widetilde{w^{0}}=w^{0}=u_{0}$ (this implies $\|e^{0}\|=\|\epsilon^{0}\|=0$), we obtain
\begin{multline*}
\|\widetilde{\epsilon^{l}}\|^{2}+\sum_{n=0}^{l-1}\|\epsilon^{n+1}-\widetilde{\epsilon^{n+1}}\|^{2}
+\frac{1}{2}\sum_{n=0}^{l-1}\|\widetilde{\epsilon^{n+1}}-e^{n}\|^{2}+\sum_{n=0}^{l-1}\nu\triangle t\|\nabla\widetilde{\epsilon^{n+1}}\|^{2}\\
\leq\sum_{n=0}^{l-l} C\triangle t\|\widetilde{\epsilon^{n}}\|^{2}+2(\triangle t)^{2}\sum_{n=0}^{l-1}\|\nabla p(t_{n+1})\|^{2}+C(\triangle t)^{2}(\int_{t_{0}}^{t_{l}}t\|u_{tt}\|_{-1}^{2}dt+\int_{t_{0}}^{t_{l}}\|u_{t}\|^{2}dt) +C\delta_{\max}^2\\
\leq\sum_{n=0}^{l-1} C\triangle t\|\widetilde{\epsilon^{n}}\|^{2}+C\triangle t+C\delta_{\max}^2.
\end{multline*}
Applying the discrete Gronwall inequality yields \eqref{firstEst}.
\end{proof}

Now, we will use the result of the lemma and improve the predicted order of
convergence for the velocity. The main result in this section is the
following theorem, stating that all $\widetilde{w^{n+1}}$, $w^{n+1}$ and $%
u^{n+1}$ are first-order approximations to the Navier-Stokes solution.

\begin{theorem}\label{ThMain}
Assume the solution to the Navier-Stokes system satisfies \eqref{A} and $%
\chi=\chi_{0}\triangle t$. Suppose $\partial\Omega \in C^{1,1}$ or $\Omega$
is convex. It holds
\begin{equation}  \label{eqTh_vel}
\triangle t\sum_{n=1}^{l}(\|\widetilde{\epsilon^{n}}\|^{2}+\|\epsilon^{n}%
\|^{2}+\|e^{n}\|^{2})\leq C((\triangle t)^{2} +\delta_{\max}^4).
\end{equation}
Additionally assume $\int_{0}^{T}\|\nabla p_{t}\|^2 \le C\, $ and the
filtering radius is bounded as $\delta_{\max}^4\le C\,\triangle t$, then $%
p^{n}$ is an approximation to $p(t_{n})$ in $L^2(\Omega)/R$ in the following
sense:
\begin{equation}
\triangle t \sum^{l}_{n=1} \| p^{n} - p(t_{n}) \|^{2} \le C(\triangle t+
\delta_{\max}^2).  \label{eqTheorem_pres}
\end{equation}
\end{theorem}

\begin{proof}
Literally reaping the arguments from \cite{JS}, pp. 66-69, one shows the estimate
\begin{multline}
\|\epsilon^{n+1}\|_{V'}^{2}-\|e^{n}\|_{V'}^{2}+\|\epsilon^{n+1}-e^{n}\|_{V'}^{2}+\nu\triangle t\|\epsilon^{n+1}\|^{2}\leq C\Big(\triangle t\|\epsilon^{n+1}\|_{V'}^{2}\\
+(\triangle t)^{2}\int_{t_{n}}^{t_{n+1}}(t\|u_{tt}\|_{-1}^{2}+\|u_{t}\|^{2})dt+(\triangle t)^{2}\|\nabla\widetilde{\epsilon^{n+1}}\|^{2}+\triangle t\|\widetilde{\epsilon^{n+1}}-e^{n}\|^{2}+\triangle t\|\epsilon^{n+1}-\widetilde{\epsilon^{n+1}}\|^{2}\Big).\label{err3}
\end{multline}
The estimate \eqref{estFilt2} gives $\| F {\epsilon^{n}}\|_{V'}\le\|\epsilon^{n}\|_{V'}+\delta_{\max}^2\|\nabla \epsilon^{n}\|$. Here and in the rest of the proof the filtering is based on the $w^{n+1}$ velocity, that is  $F\cdot:= F(w^{n+1})\cdot$.  Due to the assumption $\dO\in C^{1,1}$ or $\Omega$ is convex, the $L^2$ projection on $H$ is $H^1$ stable,
i.e. $\|\nabla \epsilon^{n}\|\le C \|\nabla\widetilde{\epsilon^{n}}\|$ and therefore we conclude
\[
\| F {\epsilon^{n}}\|_{V'}\le\|\epsilon^{n}\|_{V'}+ C\delta_{\max}^2\|\nabla \widetilde{\epsilon^{n}}\|.
\]
Using this  and \eqref{estFilt2}, we get from \eqref{et} for $\chi=\chi_0\triangle t $
\begin{align*}
\|e^{n}\|_{V'}
&=(1-\chi)\|\epsilon^{n}\|_{V'}+\chi\| F {\epsilon^{n}}\|_{V'}+\chi\|u(t_{n})- F {u(t_{n})}\|_{V'}
\leq \|\epsilon^{n}\|_{V'}+C\triangle t \left(\delta_{\max}^2\|\nabla \widetilde{\epsilon^{n}}\|+\|u(t_{n})- F {u(t_{n})}\|_{V'}\right)\\
&\le \|\epsilon^{n}\|_{V'}+C\triangle t \delta_{\max}^2\left(\|\nabla \widetilde{\epsilon^{n}}\|+ 1\right).
\end{align*}
Squaring the inequality, we get after elementary calculations
\[
\|e^{n}\|_{V'}^2\le(1+\triangle t )\|\epsilon^{n}\|_{V'}^2+C\triangle t \delta_{\max}^4\left(\|\nabla \widetilde{\epsilon^{n}}\|^2+ 1\right).
\]
We substitute the above estimate to the left-hand side of \eqref{err3} and arrive at
\begin{multline*}
\|\epsilon^{n+1}\|_{V'}^{2}-\|\epsilon^{n}\|_{V'}^{2}
+\|\epsilon^{n+1}-e^{n}\|_{V'}^{2}+\nu\triangle t\|\epsilon^{n+1}\|^{2}\\
\leq C\Big( \triangle t (\|\epsilon^{n+1}\|_{V'}^{2}+\|\epsilon^{n}\|_{V'}^{2})+(\triangle t)^{2}\int_{t_{n}}^{t_{n+1}}(t\|u_{tt}\|_{-1}^{2}+\|u_{t}\|^{2})dt
+(\triangle t)^{2}\|\nabla\widetilde{\epsilon^{n+1}}\|^{2}\\ +\triangle t(\|\widetilde{\epsilon^{n+1}}-e^{n}\|^{2}+\|\epsilon^{n+1}-\widetilde{\epsilon^{n+1}}\|^{2})
+
\triangle t \delta_{\max}^4(1+\|\nabla \widetilde{\epsilon^{n}}\|^2)\Big).
\end{multline*}
Assume for the sake of convenience $\delta_{\max}\le C$. Summing up the inequalities for $n=0,\dots,l-1$, we get
\begin{multline}
\|\epsilon^{l}\|_{V'}^{2}+\sum_{n=0}^{l-1}\|\epsilon^{n+1}-e^{n}\|_{V'}^{2}+\sum_{n=0}^{l-1}\nu\triangle t\|\epsilon^{n+1}\|^{2}\\
\leq C\left(\sum_{n=0}^{l-1}\triangle t\|\epsilon^{n+1}\|_{V'}^{2}+(\triangle t)^{2}\int_{t_{0}}^{t_{l}}(\|u_{tt}\|_{V'}^{2}+\|u_{t}\|^{2})dt
+\delta_{\max}^4\sum_{n=0}^{l-1}\triangle t\|\nabla\widetilde{\epsilon^{n}}\|^{2}\right.\\\left.+\sum_{n=0}^{l-1}\triangle t\|\widetilde{\epsilon^{n+1}}-e^{n}\|^{2}+\sum_{n=0}^{l-1}\triangle t\|\epsilon^{n+1}-\widetilde{\epsilon^{n+1}}\|^{2}
 + \triangle t\delta_{\max}^4\right).
\label{err4}
\end{multline}
Now we use the result of the Lemma~\ref{lemma_est1} to bound
\begin{equation*}
\triangle t\|\epsilon^{l}\|_{V'}^{2}+\delta_{\max}^4\sum_{n=0}^{l-1}\triangle t\|\nabla\widetilde{\epsilon^{n+1}}\|^{2}+\sum_{n=0}^{l-1}\triangle t\|\widetilde{\epsilon^{n+1}}-e^{n}\|^{2}+\sum_{n=0}^{l-1}\triangle t\|\epsilon^{n+1}-\widetilde{\epsilon^{n+1}}\|^{2}\leq C((\triangle t)^{2}+\triangle t\delta_{\max}^2+\delta_{\max}^4).
\end{equation*}
Thus, applying the Gronwall inequality to \eqref{err4}  yields
\begin{equation}\label{est_aux1}
\|\epsilon^{l}\|_{V'}^{2}+\sum_{n=0}^{l-1}\|\epsilon^{n+1}-e^{n}\|_{V'}^{2}+\sum_{n=0}^{l-1}\nu\triangle t\|\epsilon^{n+1}\|^{2} \leq C((\triangle t)^{2}+\delta^{4}_{\max}).
\end{equation}
Here we also used $\triangle t\delta^{2}_{\max}\le (\triangle t)^2+\delta^{4}_{\max}$.
Finally, the Lemma~\ref{lemma_est1} helps us to estimate
\begin{align*}
\triangle t\sum_{n=0}^{l-1}\|\widetilde{\epsilon^{n+1}}\|^{2}&\leq \triangle t\sum_{n=0}^{l-1}\|\epsilon^{n+1}-\widetilde{\epsilon^{n+1}}\|^{2}+\triangle t\sum_{n=0}^{l-1}\|\epsilon^{n+1}\|^{2}\leq  C((\triangle t)^{2}+\delta^{4}_{\max}).\\
\triangle t\sum_{n=0}^{l}\|e^{n}\|^{2}&\leq \triangle t\sum_{n=0}^{l-1}\|\epsilon^{n+1}-e^{n}\|^{2}+\triangle t\sum_{n=0}^{l-1}\|\epsilon^{n+1}\|^{2}\leq  C((\triangle t)^{2}+\delta^{4}_{\max}).
\end{align*}
These estimates together with \eqref{est_aux1} proves the velocity error estimate of the theorem.

Further we show that the pressure is weakly $\frac12$ order convergent  to the true solution.
Denote the pressure error as  $q^n=p^{n} - p(t_{n})$. We may assume $(q^n,1)=0$. It holds
\begin{equation}
-\nabla q^{n+1}=-\frac{1}{\triangle t}(\epsilon^{n+1}-e^{n})+\nu\triangle\widetilde{\epsilon^{n+1}}+(w^{n}\cdot\nabla)\widetilde{w^{n+1}}-(u(t_{n+1})\cdot\nabla)u(t_{n+1})
+R^{n}.\label{pre1}
\end{equation}
Repeating the arguments from \cite{JS} and using the Ne\v{c}as inequality, see \cite{Necas}, one deduces from   \eqref{pre1}
\[
\|q^{n+1}\|\leq c\sup_{v\in H_{0}^{1}(\Omega)^{3}}\frac{(\nabla q^{n+1},v)}{\|\nabla v\|} \leq \frac{1}{\triangle t}\|\epsilon^{n+1}-e^{n}\|_{-1}+C(\|R^{n}\|_{-1}+\|\nabla\widetilde{\epsilon^{n+1}}\|+\|\nabla{\epsilon^{n+1}}\|+\|u(t_{n+1})-u(t_{n})\|).
\]
Therefore, by using \eqref{firstEst}, we get
\begin{equation}
\triangle t\sum_{n=0}^{l-1}\|q^{n+1}\|^2\leq\frac{1}{\triangle t}\sum_{n=0}^{l-1}\|\nabla(\epsilon^{n+1}-e^{n})\|_{-1}^{2}+C(\triangle t +  \delta_{max}^2).\label{preq}
\end{equation}
To bound the first term on the right-hand side of \eqref{preq} one estimates:
\begin{equation}\label{H_1est}
\|{\epsilon^{n+1}}-e^{n}\|_{-1}\leq c\|{\epsilon^{n+1}}-e^{n}\|\leq
c(\|{\epsilon^{n+1}}-{\epsilon^{n}}\|+\|{\epsilon^{n}}-e^{n}\|)
\le c(\|\widetilde{\epsilon^{n+1}}-\widetilde{\epsilon^{n}}\|+\|{\epsilon^{n}}-e^{n}\|).
\end{equation}
The estimate for the second term on the right-hand side of \eqref{H_1est} follows from \eqref{et}:
\[
\|\epsilon^{n}-e^{n}\| \le \chi_0\triangle t(\|\epsilon^{n}-F {\epsilon^{n}}\|+\|u(t_{n})- F {u(t_{n})}\|)
\le\chi_0\triangle t(\|\epsilon^{n}\|+\|F {\epsilon^{n}}\|+\|u(t_{n})- F {u(t_{n})}\|).
\]
Thanks to \eqref{estFilt1}, \eqref{firstEst},  and $\|F {\epsilon^{n}}\|\le \|{\epsilon^{n}}\|$ we continue the above estimate as
\begin{equation}\label{H_1est2}
\|\epsilon^{n}-e^{n}\| \le C ((\triangle t)^{\frac32} + \triangle t \delta_{\max}).
\end{equation}
Below we shall prove the bound
\begin{equation*}
\sum_{n=0}^{l-1}\|\widetilde{\epsilon^{n+1}}-\widetilde{\epsilon^{n}}\|^{2}\leq  C((\triangle t)^2+ \triangle t \delta_{\max}^2).
\end{equation*}
From \eqref{step1} and \eqref{step4} we get
\begin{equation}
\frac{1}{\triangle t}(\widetilde{\epsilon^{n+1}}-e^{n})-\nu\Delta\widetilde{\epsilon^{n+1}}+\nabla p(t_{n+1})+(w^{n}\cdot\nabla)\widetilde{w^{n+1}}-(u(t_{n+1})\cdot\nabla)u(t_{n+1})=R^{n}. \label{pre2}
\end{equation}
The projection step \eqref{step2} gives
$
\epsilon^{n}=\widetilde{\epsilon^{n}}+\triangle t\nabla p^{n},
$
so \eqref{et} yields
\begin{equation*}
e^{n}=(1-\chi)(\widetilde{\epsilon^{n}}+\triangle t\nabla p^{n})+\chi F {\epsilon^{n}}+\chi(u(t_{n})- F {u(t_{n})}).
\end{equation*}
Substituting this in \eqref{pre2} implies
\begin{multline}
\frac{1}{\triangle t}(\widetilde{\epsilon^{n+1}}-\widetilde{\epsilon^{n}}) -\nu\Delta\widetilde{\epsilon^{n+1}}+(1-\chi)\nabla (p(t_{n+1})-p^{n})
+\chi\nabla p(t_{n+1})-\frac{\chi}{\triangle t}( F \epsilon^{n}-\widetilde{\epsilon^{n}})-\frac{\chi}{\triangle t}(u(t_{n})- F {u(t_{n})}) \\ +(w^{n}\cdot\nabla)\widetilde{w^{n+1}}-(u(t_{n+1})\cdot\nabla)u(t_{n+1})=R^{n}.\label{pre5}
\end{multline}
The inner product of \eqref{pre5} with $\triangle t(\widetilde{\epsilon^{n+1}}-\widetilde{\epsilon^{n}})$ gives
\begin{multline}\label{est58}
\|\widetilde{\epsilon^{n+1}}-\widetilde{\epsilon^{n}}\|^{2}+\frac{\nu\triangle t}{2}(\|\nabla\widetilde{\epsilon^{n+1}}\|^{2}-\|\nabla\widetilde{\epsilon^{n}}\|^{2} +\|\nabla(\widetilde{\epsilon^{n+1}}-\widetilde{\epsilon^{n}})\|^{2})
\\
=\triangle t(R^{n},\widetilde{\epsilon^{n+1}}-\widetilde{\epsilon^{n}})+(1-\chi)\triangle t(p(t_{n+1})-p^{n},\Div(\widetilde{\epsilon^{n+1}}-\widetilde{\epsilon^{n}}))
+\triangle t ((w^{n}\cdot\nabla)\widetilde{w^{n+1}}-(u(t_{n+1})\cdot\nabla)u(t_{n+1}),\widetilde{\epsilon^{n+1}}-\widetilde{\epsilon^{n}})
\\
-\chi\triangle t(\nabla p(t_{n+1}),\widetilde{\epsilon^{n+1}}-\widetilde{\epsilon^{n}})+\chi( F {\epsilon^{n}}-\widetilde{\epsilon^{n}},\widetilde{\epsilon^{n+1}}-\widetilde{\epsilon^{n}})+\chi(u(t_{n})- F {u(t_{n})},\widetilde{\epsilon^{n+1}}-\widetilde{\epsilon^{n}})
\\
=\triangle t(R^{n},\widetilde{\epsilon^{n+1}}-\widetilde{\epsilon^{n}})+(1-\chi)\triangle t\left[(q^{n},\Div(\widetilde{\epsilon^{n+1}}-\widetilde{\epsilon^{n}})) +(p(t_{n+1})-p(t_{n}),\Div(\widetilde{\epsilon^{n+1}}-\widetilde{\epsilon^{n}}))\right]
\\
-\chi\left[\triangle t(\nabla p(t_{n+1}),\widetilde{\epsilon^{n+1}}-\widetilde{\epsilon^{n}}) -( F {\epsilon^{n}}-\widetilde{\epsilon^{n}},\widetilde{\epsilon^{n+1}}-\widetilde{\epsilon^{n}}) -(u(t_{n})- F {u(t_{n})},\widetilde{\epsilon^{n+1}}-\widetilde{\epsilon^{n}})\right]
\\ +\triangle t ((w^{n}\cdot\nabla)\widetilde{w^{n+1}}-(u(t_{n+1})\cdot\nabla)u(t_{n+1}),\widetilde{\epsilon^{n+1}}-\widetilde{\epsilon^{n}})
\\
=I_{1}+I_{2}+I_{3}+I_{4}+I_{5}+I_{6}+I_{7}.
\end{multline}
The last term $I_7$ is estimated in \cite{OST}:
\begin{multline*}
\triangle t |((w^{n}\cdot\nabla)\widetilde{w^{n+1}}-(u(t_{n+1})\cdot\nabla)u(t_{n+1}),\widetilde{\epsilon^{n+1}}-\widetilde{\epsilon^{n}})|\\
\le \sigma\|\widetilde\epsilon^{n+1}- \widetilde\epsilon^{n}\|^2+C((\triangle t)^2 \|\widetilde\epsilon^{n+1}\|^2
+(\triangle t)^2 \|\epsilon^{n+1}\|^2+\triangle t^{\frac32} \|\nabla\epsilon^{n}\|^2\|\nabla\widetilde\epsilon^{n+1}\|^2
+\frac{\nu\triangle t}{2}\|\nabla(\widetilde{\epsilon^{n+1}}-\widetilde{\epsilon^{n}})\|^{2}+(\triangle t)^{3})
\end{multline*}
for some $\sigma>0$, which can be taken sufficiently small. Applying \eqref{firstEst} and $\|\nabla\epsilon^{n}\|\le C \|\widetilde{\nabla\epsilon^{n}}\|$ leads to
\begin{equation}\label{I8}
I_7\le \sigma\|\widetilde\epsilon^{n+1}- \widetilde\epsilon^{n}\|^2+C((\triangle t)^{3} + (\triangle t)^2 \delta^2_{\max})+(\triangle t)^{\frac32} \|\nabla\widetilde{\epsilon^{n}}\|^2\|\nabla\widetilde\epsilon^{n+1}\|^2
+\frac{\nu\triangle t}{2}\|\nabla(\widetilde{\epsilon^{n+1}}-\widetilde{\epsilon^{n}})\|^{2}.
\end{equation}
For $I_{4}$, $I_{5}$,  and $I_{6}$
one has
\begin{align}\label{est69}
I_{4}&=-\chi\triangle t(\nabla p(t_{n+1}),\widetilde{\epsilon^{n+1}}-\widetilde{\epsilon^{n}})\leq C\chi^{2}(\triangle t)^{2}\|\nabla p(t_{n+1})\|^{2}+\sigma\|\widetilde{\epsilon^{n+1}}-\widetilde{\epsilon^{n}}\|^{2},
\\
I_{5}&=\chi( F {\epsilon^{n}}-\widetilde{\epsilon^{n}},\widetilde{\epsilon^{n+1}}-\widetilde{\epsilon^{n}}) \leq C\chi^{2}(\| F {\epsilon^{n}}\|^{2}+\|\widetilde{\epsilon^{n}}\|^2)+\sigma\|\widetilde{\epsilon^{n+1}}-\widetilde{\epsilon^{n}}\|^{2}\nonumber \\ &\leq C((\triangle t)^{3} + (\triangle t)^2 \delta^2_{\max}) +\sigma\|\widetilde{\epsilon^{n+1}}-\widetilde{\epsilon^{n}}\|^{2},
\\
I_{6}&=\chi(u(t_{n})- F {u(t_{n})},\widetilde{\epsilon^{n+1}}-\widetilde{\epsilon^{n}})\leq C (\triangle t)^2 \delta^4_{\max} +\sigma\|\widetilde{\epsilon^{n+1}}-\widetilde{\epsilon^{n}}\|^{2}.\label{pre9}
\end{align}
The terms $I_{1}$, $I_{2}$ and $I_{3}$ are estimated in \cite{JS}. Using those estimates and \eqref{I8}--\eqref{pre9} in \eqref{est58} yields for sufficiently small $\sigma>0$:
\begin{multline} \label{aux2}
\|\widetilde{\epsilon^{n+1}}-\widetilde{\epsilon^{n}}\|^{2}+\frac{\nu\triangle t}{2}(\|\nabla\widetilde{\epsilon^{n+1}}\|^{2}-\|\nabla\widetilde{\epsilon^{n}}\|^{2})+(1-\chi)(\triangle t)^{2}(\|\nabla q^{n+1}\|^{2}-\|\nabla q^{n}\|^{2})
\\
\leq C\left\{(\triangle t)^{2}\int_{t_{n}}^{t_{n+1}}\|u_{tt}\|^{2}dt+(\triangle t)^{2}\int_{t_{n}}^{t_{n+1}}\|\nabla p_{t}\|^{2}dt+(\triangle t)^{4}\|\nabla p(t_{n+1})\|^{2}\right.\\\left.+(\triangle t)^{3} + (\triangle t)^2 \delta^2_{\max}+ \triangle t^{\frac32} \|\nabla\widetilde{\epsilon^{n}}\|^2\|\nabla\widetilde\epsilon^{n+1}\|^2\right\}.
\end{multline}
We sum up the estimate for $n=0,\dots,l-1$ and apply our assumptions for the solution to  Navier-Stokes solution. This leads to the bound
\begin{equation*}
\sum_{n=0}^{l-1}\|\widetilde{\epsilon^{n+1}}-\widetilde{\epsilon^{n}}\|^{2}
+\frac{\nu\triangle t}{2}\|\nabla\widetilde{\epsilon^{l}}\|^{2}
\leq C((\triangle t)^{2}+ \triangle t \delta_{\max}^4+ (\triangle t)^{\frac32} \sum_{n=0}^{l-1} \|\nabla\widetilde{\epsilon^{n}}\|^2\|\nabla\widetilde\epsilon^{n+1}\|^2).
\end{equation*}
The application of the discrete Gronwall inequality, \eqref{firstEst} and the assumption $\delta^4_{\max}\le C\triangle t$ yields
\begin{align*}
\sum_{n=0}^{l-1}\|\widetilde{\epsilon^{n+1}}-\widetilde{\epsilon^{n}}\|^{2}
+\frac{\nu\triangle t}{2}\|\nabla\widetilde{\epsilon^{l}}\|^{2}&
\le C\,((\triangle t)^2+ \triangle t \delta_{\max}^2)\,\exp\left\{(\triangle t)^{\frac12}\sum_{n=0}^{l-1}\|\nabla\widetilde{\epsilon^{n+1}}\|^2\right\}
\\
&\le C\,((\triangle t)^2+ \triangle t \delta_{\max}^2)\,\exp\left\{C((\triangle t)^{\frac12}+(\triangle t)^{-\frac12}\delta_{\max}^4)\right\}
\\
&\le C((\triangle t)^2+ \triangle t \delta_{\max}^4).
\end{align*}
Therefore, \eqref{preq}--\eqref{H_1est2} yield the desired bound:
\begin{equation*}
\triangle t\sum_{n=0}^{l-1}\|q^{n+1}\|^2\leq C(\triangle t+ \delta_{\max}^2).
\end{equation*}
\end{proof}

\end{document}